\documentclass[12pt, reqno]{amsart}

\usepackage[margin=1.1in]{geometry} 

\usepackage{amssymb,upref}
\usepackage{amsmath, amsthm}
\usepackage{enumerate, graphicx, float}
\usepackage{relsize, mathrsfs, upgreek, bbm, eufrak}

\usepackage{color}
\usepackage{esint}

\theoremstyle{plain}

\newtheorem{thm}{Theorem}[section]

\newtheorem{coro}[thm]{Corollary}

\theoremstyle{remark}

\newtheorem{remark}[thm]{Remark}
\numberwithin{equation}{section}

\def\XXint#1#2#3{{\setbox0=\hbox{$#1{#2#3}{\int}$ }
\vcenter{\hbox{$#2#3$ }}\kern-.6\wd0}}

\newcommand{\bee}{\begin{equation}}
\newcommand{\eee}{\end{equation}}
\newcommand{\be}{\begin{equation*}}
\newcommand{\ee}{\end{equation*}}

\newcommand{\eps}{\varepsilon}

\newcommand{\na}{\mathbb{N}}

\newcommand{\rn}{\mathbb{R}^n}

\newcommand{\norm}[2]{\|#1\|_{#2}}

\newcommand{\Ma}{{\mathcal M}_\alpha}

\newcommand{\dx}{\, dx}

\newcommand{\Ainf}{A_\infty(\rn)}

\DeclareMathOperator*{\esssup}{ess\,sup\,}

\DeclareMathOperator*{\essinf}{ess\,inf\,}

\makeatletter
\@namedef{subjclassname@2020}{\textup{2020} Mathematics Subject Classification}
\makeatother

\begin{document}



\address{Diego Maldonado, Kansas State University, Department of Mathematics. 138 Cardwell Hall, Manhattan, KS-66506, USA.} \email{dmaldona@ksu.edu}

\address{Javier Soria, Departamento de An\'alisis Matem\'atico y Matem\'atica Aplicada, Fa\-cul\-tad de Matem\'aticas, Universidad Complutense de Madrid, Plaza de Ciencias 3, 28040 Madrid, Spain and Instituto de Ciencias Matem\'aticas ICMAT (CSIC-UAM-UC3M-UCM), C/ Nicol\'as Cabrera, 13--15, 28049 Madrid, Spain.}
\email{javier.soria@ucm.es}

\title[Weighted inequalities]{On weighted estimates for fractional operators and applications to Hardy-type inequalities}
\author[Diego Maldonado and Javier Soria]{Diego Maldonado$^{*}$ and Javier Soria$^{**}$}

\thanks{$^{*}$Partially supported by Simons Foundation grant MPS-TSM-00007229.$^{**}$Partially supported by grants PID2020-113048GB-I00, PID2024-155917NB-I00, and CEX2019-000904-S, funded by MCIN/AEI/ 10.13039/501100011033, and Grupo UCM-970966.}

\date{\today}

\begin{abstract} New necessary and sufficient conditions for two-weight estimates for fractional maximal and integral operators are established by means of Harnack-type estimates for Muckenhoupt weights. Applications include weighted Hardy inequalities with weights based on the Hardy-Littlewood maximal operator. 
\end{abstract}

\maketitle

\section{Introduction and main results}\label{sec:main:results}

For $0 < \alpha < n$ the fractional maximal operator $\Ma$ is defined as
\begin{equation*}
\Ma(f)(x) := \sup\limits_{Q \ni x} |Q|^{\alpha/n} \fint_Q |f(y)| \, dy,
\end{equation*}
where $Q \subset \rn$ is a cube (always with sides parallel to the coordinate axes) and the notation $\fint_E u$ stands for $\frac{1}{|E|}\int_E u(x) \, dx$; that is, the average of $u$, with respect to Lebesgue measure, over a set $E$ with Lebesgue measure $|E| >0$ (while the notation $u(E)$ will stand for $\int_E u(x) \, dx$). Similarly, $I_\alpha$ denotes the fractional integral operator
$$
I_\alpha(f)(x):= \int_{\rn} \frac{f(y)}{|x-y|^{n-\alpha}}\, dy. 
$$
When $\alpha=0$, $\mathcal{M}_0 =\mathcal{M}$ is the Hardy-Littlewood maximal operator. The literature on two-weight inequalities for $\Ma$ and $I_\alpha$ is vast, but we will base the statements of our main results on the following classical theorem by E.~Sawyer and R.~Wheeden (see  \cite[Theorem~1]{SW1}): For $1 < p \leq q < \infty$ and $0 < \alpha < n$ define
\begin{equation}\label{def:alpha0}
\Theta(\alpha):= \alpha + n \left( \frac{1}{q} - \frac{1}{p}\right),
\end{equation}
then, given nonnegative measurable functions $u$ and $v$ and $r  > 1$, the condition
\begin{equation}\label{SW:p:q:r}
[(u,v)]_{A_{p,q}^{r, \alpha}}:=\sup\limits_Q |Q|^{\Theta(\alpha)/n}\left(\fint_Q u^r  \right)^\frac{1}{rq} \left(\fint_Q  v^{-rp'/p}\right)^\frac{1}{rp'} < \infty
\end{equation}
implies the two-weight strong-type $(p,q)$ for $I_\alpha$
\begin{equation}\label{p:q:I:alpha}
\left(\int_{\rn} |I_\alpha(f)(x)|^q  u(x) \dx \right)^{1/q}\leq C_{n,p,q,\alpha}  [(u,v)]_{A_{p,q}^{r, \alpha}}  \left(\int_{\rn} |f(x)|^p v(x) \dx \right)^{1/p}
\end{equation}
for every measurable function $f$. Recall that for $1 < p < \infty$ a weight $w$ in $\rn$ (that is, $w \in L^1_{\rm{loc}}(\rn)$ with $w \geq 0$ a.e. in $\rn$) is said to belong to the \emph{Muckenhoupt class} $A_p(\rn)$ if 
\begin{equation}\tag*{$(A_p)$}\label{Ap}
[w]_{A_p}:= \sup\limits_{Q} \left(\fint_Q w \right)  \left(\fint_Q w^{\frac{-1}{(p-1)}}  \right)^{p-1} < \infty
\end{equation}
where $Q \subset \rn$ is a cube. The endpoint classes for $p=1$ and $p=\infty$ are defined by
\begin{equation}\tag*{$(A_1)$}\label{def:A1}
[w]_{A_1}:= \sup\limits_{Q} \left(\fint_Q w \right)  \left(\essinf_Q w \right)^{-1} < \infty
\end{equation}
and
\begin{equation}\tag*{$(A_\infty)$}\label{def:Ainf}
[w]_{A_\infty}:= \sup\limits_{Q} \left(\fint_Q w \right)  \exp \left(-\fint_Q \ln w  \right) < \infty.
\end{equation}
As it turns out, $
A_\infty(\rn) = \bigcup_{p\geq1}A_p(\rn)
$
(see for instance \cite[Section~9.3]{GrafakosBook}).

Throughout this article, our basic standing assumption on the pair $(u,v)$ will be that both $u$ and $v^{-p'/p}$ belong to $\Ainf$. In this case, there are constants $N \geq 1$ and $r > 1$, depending only on $[u]_{\Ainf}$, $[v^{-p'/p}]_{\Ainf}$, and $n$, such that  
$$
\left(\fint_Q u^r  \right)^{1/r} \leq N  \left(\fint_Q u  \right) \quad \text{and} \quad \left(\fint_Q  v^{-rp'/p}\right)^{1/r} \leq N \left(\fint_Q  v^{-p'/p}\right), \quad \forall Q,
$$ 
(see for instance \cite[Theorem~7.2.5]{GrafakosBook}) and then \eqref{SW:p:q:r} becomes equivalent to 
\begin{equation}\label{SW:p:q}
[(u,v)]_{A_{p,q}^{\alpha}}:=\sup\limits_Q |Q|^{\Theta(\alpha)/n}\left(\fint_Q u  \right)^{1/q} \left(\fint_Q  v^{-p'/p}\right)^{1/p'} < \infty
\end{equation}
with the estimate $[(u,v)]_{A_{p,q}^{\alpha}} \leq [(u,v)]_{A_{p,q}^{r, \alpha}} \leq N^{1/q + 1/p'} [(u,v)]_{A_{p,q}^{\alpha}}.$ Since $1< p \leq q < \infty$ and $0 < \alpha < n$ we always have $\Theta(\alpha) \leq \alpha <n$. Moreover, we may (and do) assume that $\Theta(\alpha) \geq 0$. Otherwise, Lebesgue's Differentiation Theorem applied in \eqref{SW:p:q} would imply $u \equiv 0$ a.e.

To compare $\Ma$ and $I_\alpha$ when $0 < \alpha < n$, a simple computation yields $\Ma(f)(x) \leq 2^{n-\alpha} I_\alpha(|f|)(x)$ for every measurable $f$ and $x \in \rn$. On the other hand, as B.~Muckenhoupt and R.~Wheeden proved in  \cite[Theorem~1]{MuWheeden}, for every $0 < p < \infty$, $0 < \alpha < n$, and $u \in \Ainf$, there exists $C >0$, depending only on $p$, $\alpha$, $n$, and $[u]_{\Ainf}$ such that
$$
\int_{\rn} |I_\alpha(f)(x)|^p u(x) \dx \leq C \int_{\rn} |\mathcal{M}_\alpha(f)(x)|^p u(x) \dx
$$
for every measurable $f$. Moreover, by \cite[Theorem~1]{SW1}  the inequality \eqref{SW:p:q} turns out to be necessary for \eqref{p:q:I:alpha} to hold (and this does not require $u, v^{-p'/p} \in \Ainf$). Consequently, under the assumption $u, v^{-p'/p} \in \Ainf$ the condition \eqref{SW:p:q} is both necessary and sufficient for the inequality \eqref{p:q:I:alpha} to hold true for either $I_\alpha$ or $\mathcal{M}_\alpha$. This is also true in the case $\alpha =0$ (which implies $\Theta(\alpha) =0$ and hence $p=q$), since the necessity of \eqref{SW:p:q} (with $\Theta(\alpha) =0$ and $p=q$) for the $(u,v)$-weighted strong $(p,p)$-type of the Hardy-Littlewood maximal function has been established by B. Muckenhoupt in \cite[Theorem 1]{Mucken1} and its sufficiency, when $v^{-p'/p} \in \Ainf$, by C. J. Neugebauer in \cite[Theorem 5]{Neuge1}.

The purpose of this note is to provide new, simpler necessary and sufficient conditions for  \eqref{SW:p:q} (and hence for \eqref{p:q:I:alpha}) to hold under the hypothesis $u, v^{-p'/p} \in \Ainf$. Based on such sufficient conditions, examples of pairs $(u,v)$ satisfying \eqref{SW:p:q} will be constructed and then used in several applications to Hardy-type inequalities.

Let us now describe our main results according to the cases $\Theta(\alpha) = 0$ and $\Theta(\alpha) > 0$. When $\Theta(\alpha) = 0$ the Lebesgue Differentiation Theorem again, now applied to \eqref{SW:p:q}, yields the pointwise inequality $u(x)^{1/q} \leq [(u,v)]_{A_{p,q}^{\alpha}} v(x)^{1/p}$ for a.e. $x \in \rn$.  Our first result shows that, conversely, if $u, v^{-p'/p} \in \Ainf$ then the pointwise condition $u^{1/q} \leq C v^{1/p}$ implies \eqref{SW:p:q}. More precisely, 

\begin{thm}\label{thm:SW:Ainf:p:q:alpha0=0} Fix $1 < p \leq q < \infty$ and $0 < \alpha < n$ (with $\alpha=0$ allowed for $\mathcal{M}_0$) such that $\Theta(\alpha) = 0$ (that is, $1/q = 1/p - \alpha/n$). Given $(u,v)$ with $u, v^{-p'/p} \in \Ainf$ suppose that there exists $C_0 > 0$ with
\begin{equation}\label{u:1/q:C:v:1/p}
u(x)^{1/q} \leq C_0 \, v(x)^{1/p}, \quad \text{a.e. } x \in \rn.
\end{equation}
Then \eqref{SW:p:q} holds true and consequently 
\begin{equation*}
\left(\int_{\rn} |I_\alpha(f)(x)|^q  u(x) \dx \right)^{1/q}\leq C_0 C_1 \left(\int_{\rn} |f(x)|^p v(x) \dx \right)^{1/p},
\end{equation*}
for every measurable $f$ (with $I_\alpha$ replaced by $\mathcal{M}_0$ if $\alpha =0$), where $C_1 > 0$ depends only on $n$, $p$, $q$, $\alpha$, $[u]_{\Ainf}$, and $[v^{-p'/p}]_{\Ainf}$.
\end{thm}

For the case $\Theta(\alpha) > 0$ (that is, $1/q > 1/p - \alpha/n$) we prove the following

\begin{thm}\label{thm:SW:Ainf:p:q} Fix $1 < p \leq q < \infty$ and $0 < \alpha < n$ such that $\Theta(\alpha) > 0$. Given $(u,v)$ with $u, v^{-p'/p} \in \Ainf$ the condition $u^{1/q}/v^{1/p} \in L^{n/\Theta(\alpha), \infty}(\rn)$ implies \eqref{SW:p:q} with the estimate
\begin{equation}\label{uv:pq:thm:SW}
[(u,v)]_{A_{p,q}^{\alpha}} \leq C_2 \norm{u^{1/q}/v^{1/p}}{L^{n/\Theta(\alpha), \infty}(\rn)},
\end{equation}
where $C_2 > 0$ depends only on $n$, $\alpha$, $p$, $q$, $[u]_{\Ainf}$, and $[v^{-p'/p}]_{\Ainf}$. In particular, 
$$
\left(\int_{\rn} |I_\alpha(f)(x)|^q  u(x) \dx \right)^{1/q}\leq C_2 \norm{u^{1/q}/v^{1/p}}{L^{n/\Theta(\alpha), \infty}(\rn)} \left(\int_{\rn} |f(x)|^p v(x) \dx \right)^{1/p},
$$
for every measurable $f$. 

If, in addition, $u^{1/q}/v^{1/p} \in \Ainf$, then \eqref{SW:p:q} is equivalent to 
$$
\mathcal{M}_{\Theta(\alpha)}(u^{1/q}/v^{1/p}) \in L^\infty(\rn)
$$ 
with the estimates
$$
\theta \norm{\mathcal{M}_{\Theta(\alpha)}(u^{1/q}/v^{1/p})}{L^\infty(\rn)} \leq [(u,v)]_{A_{p,q}^{\alpha}} \leq \theta^{-1} \norm{\mathcal{M}_{\Theta(\alpha)}(u^{1/q}/v^{1/p})}{L^\infty(\rn)},
$$
where $\theta \in (0,1)$ depends only on $n$, $\alpha$, $p$, $[u]_{\Ainf}$, $[v^{-p'/p}]_{\Ainf}$, and $[u^{1/q}/v^{1/p}]_{\Ainf}$.   
\end{thm}

\begin{remark} Notice that given $1 < p \leq q < \infty$ and $0 < \alpha < n$ with $\Theta(\alpha) = 0$, by setting $t := q/p' + 1$ the condition \eqref{SW:p:q:r} follows from \eqref{u:1/q:C:v:1/p} together with the assumption $u \in A_t(\rn)$. Indeed, given $u \in A_t(\rn)$ there are constants $r > 1$ and $N \geq 1$, depending only on $[u]_{A_t(\rn)}$, $t$, and $n$, such that $u^r \in A_t(\rn)$ with $[u^r]_{A_t(\rn)} \leq N [u]_{A_t(\rn)}^r$ (see for instance \cite[Theorem 9.2.5]{GrafakosBook}). Hence,  by using that $v^{-1/p} \leq C u^{-1/q}$ a.e. from \eqref{u:1/q:C:v:1/p}, we can write
\begin{align*}
\left(\fint_Q u^r  \right)^\frac{1}{rq} \left(\fint_Q  v^{-rp'/p}\right)^\frac{1}{rp'} & \leq C_0 \left(\fint_Q u^r  \right)^\frac{1}{rq} \left(\fint_Q  u^{-rp'/q}\right)^\frac{1}{rp'}\\
& = C_0 \left[\left(\fint_Q u^r  \right) \left(\fint_Q  u^{-rp'/q}\right)^\frac{q}{p'} \right]^\frac{1}{rq} \leq C_0 [u^r]_{A_t(\rn)}^\frac{1}{rq}
\end{align*}
and \eqref{SW:p:q:r} follows with $[(u,v)]_{A_{p,q}^{r, \alpha}} \leq C_0 N^\frac{1}{rq} [u]_{A_t(\rn)}^{1/q}.$ Similarly, the inequality \eqref{u:1/q:C:v:1/p} now together with the condition $v^{-p'/p} \in A_{t'}(\rn)$ (where $(t'-1)(t-1) =1$) also quantitatively implies \eqref{SW:p:q:r}. Thus, Theorem \ref{thm:SW:Ainf:p:q:alpha0=0} allows the passage from the assumption $u \in A_t(\rn)$ or $v^{-p'/p} \in A_{t'}(\rn)$, for the precise index $t := q/p' + 1$, to $u, v^{-p'/p} \in \Ainf$. 
\end{remark}

\begin{remark} To the best of our knowledge, Theorem \ref{thm:SW:Ainf:p:q} seems to be the first to establish sufficient conditions for the $(u,v)$-based estimates \eqref{p:q:I:alpha} phrased in terms of the quotient~$u^{1/q}/v^{1/p}$. 
\end{remark}

The rest of the article is organized as follows: Section~\ref{secc:prelim} includes an account on basic properties of Muckenhoupt weights.  In Section~\ref{sec:Ma:Linfty} we identify the largest rearrangement-invariant space $D_\alpha$ so that the fractional maximal operator $\mathcal{M}_\alpha:D_\alpha\rightarrow L^\infty(\mathbb R^n)$ is bounded. This boundedness property is crucial to the proof of Theorem~\ref{thm:SW:Ainf:p:q}. In Section~\ref{secc:proof:main:thms} we prove our main results Theorems \ref{thm:SW:Ainf:p:q:alpha0=0} and \ref{thm:SW:Ainf:p:q}. In Section~\ref{secc:examples:u:Mg} we construct examples of weight pairs $(u,v)$ satisfying \eqref{SW:p:q} where $u$ takes the form $u= \mathcal{M}(g)^{\Theta(\alpha)q/n} v^{q/p}$ with $g \in L^1(\rn)$. Finally, in Section~\ref{secc:app:Hardy} we apply Theorems \ref{thm:SW:Ainf:p:q:alpha0=0} and \ref{thm:SW:Ainf:p:q} with the classes of weights constructed in Section \ref{secc:examples:u:Mg}  to produce new families of weighted Hardy inequalities, including some extensions of the classical Caffarelli-Kohn-Nirenberg inequalities (see Theorem \ref{thm:Mg:p<q:u:A1:app} and Corollary~\ref{coro:Mg:p<q:u:A1:app}). 

\section{Preliminaries}\label{secc:prelim}

Given two nonnegative quantities or functions $A$ and $B$, we will frequently use the notation $A \lesssim B$ to indicate the presence of an implied constant $C > 0$ such that $A \leq C B$ and the nature of the constant $C >0$ will be described in each case. We write $A \simeq B$ if $A \lesssim B$ and $B \lesssim A$.

\subsection{On Muckenhoupt weights}\label{sec:Ap:weights} Typical examples of $A_p$-weights are the locally integrable powers of $|x|$. More precisely, for $1 < p < \infty$, the weight $|x|^a \in A_p(\rn)$ if and only if $-n < a < n (p-1)$ and $|x|^a \in A_1(\rn)$ if and only if $-n < a \leq 0$, see for instance  \cite[p.~286]{GrafakosBook}.  In the construction of $\Ainf$ weights in Section~\ref{secc:examples:u:Mg}, we will use the facts that if $g \in L^1_{\rm{loc}}(\rn)$ with $\mathcal{M}(g)(x) < \infty$ for a.e. $x \in \rn$ and $\theta \in (0,1)$, then $\mathcal{M}(g)^\theta \in A_1(\rn)$ with the estimate
\begin{equation}\label{Mg:theta:A1}
[\mathcal{M}(g)^\theta]_{A_1(\rn)} \leq \frac{C_n}{1-\theta},
\end{equation}
for a dimensional constant $C_n >0$, see for instance \cite[Theorem~9.2.7]{GrafakosBook}. Also, if $w_1, w_2 \in A_1(\rn)$ and $\tau > 1$, then
\begin{equation}\label{w1w2tau:A1}
w_1 w_2^{-(\tau-1)} \in A_\tau(\rn),
\end{equation}
see for instance \cite[Exercise 7.1.2]{GrafakosBook}.

During the proofs of Theorems \ref{thm:SW:Ainf:p:q:alpha0=0} and \ref{thm:SW:Ainf:p:q} (in Section~\ref{secc:proof:main:thms}) we will repeatedly make use of the following characterization of $\Ainf$ established in \cite[Theorem~1.1]{Ma}. 

\begin{thm}\label{thm:Ainfty:H} Fix a weight $w$ in $\rn$ with $w(Q) > 0$ for every cube $Q \subset \rn$. Then the following are quantitatively equivalent:
\begin{enumerate}[(i)]
\item\label{w:Ainfty} $w \in A_\infty(\rn)$,
\item\label{H:every:eps} for every $\eps \in (0,1)$ there exists $C_\eps \geq 1$ such that for every cube $Q \subset \rn$ there is a set $H_\eps \subset Q$ satisfying
\begin{equation*}
\min\left\{\frac{|H_\eps|}{|Q|}, \frac{w(H_\eps)}{w(Q)}\right\} \geq (1-\eps) \quad \text{ and } \quad \esssup\limits_{H_\eps} w \leq C_\eps \essinf\limits_{H_\eps} w.
\end{equation*}
\end{enumerate}
\end{thm}

\begin{remark}\label{rmk:w:in:H} Given $w \in A_\infty(\rn)$, $\eps \in (0,1)$, a cube $Q \subset \rn$, and $H_\eps \subset Q$ as in Theorem~\ref{thm:Ainfty:H} \eqref{H:every:eps}, for a.e. $x \in H_\eps$ we can write
\begin{align*}
w(x) &\leq \esssup\limits_{H_\eps} w \leq C_\eps \essinf\limits_{H_\eps} w \leq C_\eps \fint_{H_\eps} w \leq \frac{C_\eps}{(1-\eps)} \fint_{Q} w\\
& \leq \frac{C_\eps}{(1-\eps)^2}  \fint_{H_\eps} w \leq \frac{C_\eps}{(1-\eps)^2} \esssup\limits_{H_\eps} w \leq \frac{C_\eps^2}{(1-\eps)^2}\essinf\limits_{H_\eps} w \leq \frac{C_\eps^2}{(1-\eps)^2} w(x). 
\end{align*}
That is,
\begin{equation}\label{w:like:ave:Q}
w(x) \simeq   \fint_{H_\eps} w \simeq  \fint_{Q} w \quad \text{a.e. } x \in H_\eps,
\end{equation}
where the implied constants depend only on $\eps$, $[w]_{A_\infty}$, and $n$. 
\end{remark}

\begin{remark}\label{rmk:N:weights} Fix $N$ weights $w_1, \ldots, w_N \in \Ainf$, $\eps \in (0,1)$, and a cube $Q \subset \rn$. If $H_\eps^j$ denotes the set $H_\eps \subset Q$ from Theorem~\ref{thm:Ainfty:H} \eqref{H:every:eps}, corresponding to each $w_j \in \Ainf$, $j=1, \ldots, N$, then
$$
|Q \setminus \cap_{j=1}^N H_\eps^j| = |\cup_{j=1}^N (Q \setminus H_\eps^j)| \leq \sum\limits_{j=1}^N |Q \setminus H_\eps^j| \leq \eps N |Q|,
$$
since $|Q \setminus H_\eps^j| \leq \eps |Q|$ for every $j=1, \ldots, N.$ Thus, by choosing $\eps < 1/N$, 
\begin{equation}\label{meas:H:Q}
\Biggl|\bigcap_{j=1}^N H_\eps^j\Biggr| \geq (1-\eps N) |Q|.
\end{equation}
\end{remark}

\subsection{Reverse H\"older classes} For $1 < s < \infty$ we write $w \in RH_s(\rn)$ if
\begin{equation*}
[w]_{RH_s}:= \sup\limits_{Q} \left(\fint_Q w(x)^s \dx \right)^{1/s}  \left(\fint_Q w(x) \dx \right)^{-1} < \infty.
\end{equation*}
As it turns out (see for instance \cite[Section~9.3]{GrafakosBook}),
$
\bigcup_{s > 1} RH_s(\rn) = A_\infty(\rn).
$

For $s=\infty$ we write $w \in RH_\infty(\rn)$ if 
\begin{equation*}
[w]_{RH_\infty}:= \sup\limits_{Q} \left(\esssup\limits_{Q} w\right)  \left(\fint_Q w(x) \dx \right)^{-1} < \infty.
\end{equation*}
It is a known fact that the class $RH_\infty(\rn)$ is invariant under positive powers, that is, the implication
\begin{equation}\label{RHinfty:powers}
w \in RH_\infty(\rn), \ell > 0 \Rightarrow w^\ell \in RH_\infty(\rn)
\end{equation}
holds true with $[w^\ell]_{RH_\infty(\rn)}$ depending only on $[w]_{RH_\infty(\rn)}$, $\ell$, and $n$ (see \cite[Theorem~4.2]{CUN}) and that the inclusion $RH_\infty(\rn) \subset \bigcap_{s > 1} RH_s(\rn)$ is proper (see \cite[p.~2948]{CUN}). In addition, the class $\bigcap_{s > 1} RH_s(\rn)$ constitutes the multipliers of $\Ainf$. In particular, the implication
\begin{equation}\label{multipliers:Ainfty}
w \in \Ainf, u \in \bigcap\limits_{s > 1} RH_s(\rn) \Rightarrow wu \in \Ainf
\end{equation}
holds true with $[wu]_{\Ainf} \leq [w^{\tau}]_{\Ainf}^{1/\tau} [u^{\tau'}]_{\Ainf}^{1/{\tau'}}$ for some $\tau > 1$ depending only on $[w]_{\Ainf}$ and $n$, see for instance \cite[Theorem~3.5]{JN}. 

\begin{remark}\label{comments:A1:RHinfty}  $RH_\infty(\rn)$ and $A_1(\rn)$. The classes $RH_\infty(\rn)$ and $A_1(\rn)$ are related as follows: if $w \in A_1(\rn)$ then there exists $\theta_0 > 0$ (depending only on $[w]_{A_1(\rn)}$ and $n$) such that $w^{-\theta_0} \in RH_\infty(\rn)$. Conversely, if $w \in RH_\infty(\rn)$, then there exists $\theta_1 >0$ (depending only on $[w]_{RH_\infty(\rn)}$ and $n$) such that $w^{-\theta_1} \in A_1(\rn)$, see \cite[Corollary 4.5]{CUN}. 

In particular, one has that $|x|^\delta \in RH_\infty(\rn)$ for every $\delta \geq 0$ since by fixing $0 < q < n$ the weight $|x|^{-q}$ belongs to $A_1(\rn)$, then $|x|^{\theta} \in RH_\infty(\rn)$ for some $\theta >0$. Finally, \eqref{RHinfty:powers} applied with $\ell = \delta/\theta$ yields $|x|^\delta \in RH_\infty(\rn)$. 
\end{remark}

\section{On $L^\infty(\rn)$ as the range for $\Ma$}\label{sec:Ma:Linfty}

Given $0<\alpha<n$, in this section  we will characterize the optimal domain $D_\alpha$, for the fractional maximal operator $\mathcal{M}_\alpha$, so that the boundedness $\mathcal{M}_\alpha:D_\alpha\rightarrow L^\infty(\mathbb R^n)$ holds. For this purpose, we recall the definition of the nonincreasing rearragment of $f\in L^1_{\rm loc}(\mathbb R^n)$, denoted as $f^*$ (see \cite{BeSh}):
$$
f^*(t)=\inf\{s>0:\big|\{x\in\mathbb R^n:|f(x)|>s\}\big|<t\}, \quad \forall t >0.
$$
We will also need the following result   \cite[Theorem~1.1]{CKOP},  where the authors show that if $f\in L^1_{\rm loc}(\mathbb R^n)$, then
\begin{equation}\label{mallessup}
(\mathcal{M}_\alpha f)^*(t)\lesssim \sup_{t<s<\infty}s^{\alpha/n}f^{**}(s), \quad \forall t > 0, 
\end{equation}
and, there exists  a function $g\in L^1_{\rm loc}(\mathbb R^n)$, with $g^*=f^*,$  such that
\begin{equation}\label{suplesmal}
\sup_{t<s<\infty}s^{\alpha/n}f^{**}(s)\lesssim(\mathcal{M}_\alpha g)^*(t),
\end{equation}
where  $f^{**}(t)=\frac1t\int_0^tf^*(s)\,ds$ and  the implied constants depend only on $\alpha$ and $n$.
\medskip

In view of this, it is reasonable to introduce the following rearrangement invariant structure: Given a  set $X$, satisfying that $\big(L^1(\mathbb R^n)\cap L^\infty(\mathbb R^n)\big)\subset X \subset L^1_{\rm loc}(\mathbb R^n)$, we define
$$
\widetilde{X}=\underset{Y \subset X\text{ is an r.i.\ space}}{\bigcup Y},
$$
where  $(Y,\|\cdot\|)$ is a rearrangement invariant space (r.i.)\ if it satisfies the following conditions (see \cite[Definitions I.1.1 and II.4.1]{BeSh}):
\medskip

\noindent
$Y=\{f {\text{ is measurable and }} \|f\|<\infty\}\/$ with a norm
$\|\cdot\|$ (called  Banach function norm) for which, given
 $f,g,f_n\in Y,\;A\subset \mathbb R^n\/$ measurable:
 \medskip
 
\begin{enumerate}
\item[(i)]  $\|f\|\le\|g\|\/$, if $|f|\le|g|\/$,\medskip

\item[(ii)]  $0\le f_n\le f_{n+1}\to f$, a.e. $\Rightarrow\|f_n\|\to\|f\|\/$,\medskip

\item[(iii)]  $\chi_E\in Y\/$, if $|E|<\infty\/$,\medskip

\item[(iv)]  $\displaystyle\int_E|f(x)|\,dx\le C_E\|f\|\/$, if $|E|<\infty\/$, and \medskip

\item[(v)] $\|f\|=\|g\|$, if $f^*=g^*$.
\end{enumerate}
Since, for every r.i. space $Y$, we have that \cite[II.Theorem~6.6]{BeSh},
$$
\big(L^1(\mathbb R^n)\cap L^\infty(\mathbb R^n)\big)\subset Y \subset \big(L^1(\mathbb R^n)+ L^\infty(\mathbb R^n)\big),
$$  
it is then clear that 
$$
\big(L^1(\mathbb R^n)\cap L^\infty(\mathbb R^n)\big)\subset \widetilde{X}\subset X\cap \big(L^1(\mathbb R^n)+ L^\infty(\mathbb R^n)\big).
$$
Also,  if $f\in\widetilde{X}$ and $f^*=g^*$, then $g\in \widetilde{X}$. Moreover, if $X$ is already an r.i. space, then $\widetilde{X}=X$. Let us define 
\begin{equation}\label{xalp}
X_\alpha :=\Big\{f\in L^1_{\rm loc}(\mathbb R^n):\mathcal{M}_\alpha f\in L^\infty(\mathbb R^n)\Big\}.
\end{equation}
The main result of this section is the following:

\begin{thm}\label{thm:Ma:Linfty}
For $0<\alpha<n$, let $X_\alpha$ be as in \eqref{xalp}. Then, $\widetilde{X_\alpha}=L^{n/\alpha,\infty}(\mathbb R^n)$. Hence, $L^{n/\alpha,\infty}(\mathbb R^n)$ is the largest r.i.\ space $Y$ for which $\mathcal{M}_\alpha:Y\rightarrow L^\infty(\mathbb R^n)$ is bounded.
\end{thm}

\begin{proof}
Let us first observe that $\widetilde{X_\alpha}$ is well defined; that is, $\big(L^1(\mathbb R^n)\cap L^\infty(\mathbb R^n)\big)\subset X_\alpha$. In fact, if $f\in \big(L^1(\mathbb R^n)\cap L^\infty(\mathbb R^n)\big)$, then for every cube $Q$:
$$
|Q|^{\alpha/n-1}\int_Q|f(x)|\,dx\le \|f\|_\infty, \text{ if }  |Q|<1,
$$
and
$$
|Q|^{\alpha/n-1}\int_Q|f(x)|\,dx\le \|f\|_1, \text{ if }  |Q|\ge1.
$$
Thus,  $\|\mathcal{M}_\alpha f\|_\infty\le\|f\|_{L^1(\mathbb R^n)\cap L^\infty(\mathbb R^n)}:=\max\{\|f\|_{L^1(\mathbb R^n)},\|f\|_{L^\infty(\mathbb R^n)}\}$.
\medskip

Let $h\in  L^1_{\rm loc}(\mathbb R^+)$ and define the Hardy operator  $Ah(t)=\frac1t\int_0^th(s)\,ds$. It is well known that (see \cite[p.~240]{HLP}), for $1<p\le \infty$,
$A:L^p(\mathbb R^+)\rightarrow L^p(\mathbb R^+)$  and hence, for every $1<p<\infty$, $A:L^{p,\infty}(\mathbb R^+)\rightarrow L^{p,\infty}(\mathbb R^+)$. Lets us now prove the following remark:
\begin{equation}\label{equivnorm}
f\in L^{n/\alpha,\infty}(\mathbb R^n)\iff \sup_{t<s<\infty}s^{\alpha/n}f^{**}(s)\in L^\infty(\mathbb R^+).
\end{equation}
In fact, since $f^*\le f^{**}=A(f^*)$, 
\begin{align*}
\sup_{t<s<\infty}s^{\alpha/n}f^{**}(s)\in L^\infty(\mathbb R^+)& \iff A(f^*)\in L^{n/\alpha,\infty}(\mathbb R^+)\\
&\quad\Rightarrow f^*\in L^{n/\alpha,\infty}(\mathbb R^+)\iff f\in L^{n/\alpha,\infty}(\mathbb R^n).
\end{align*}
Conversely, if  $f\in L^{n/\alpha,\infty}(\mathbb R^n)$, since $1<n/\alpha<\infty$, 
\begin{align*}
\Big\|\sup_{t<s<\infty}s^{\alpha/n}f^{**}(s)\Big\|_{L^\infty(\mathbb R^+)}&=\sup_{0<s<\infty}s^{\alpha/n}f^{**}(s)
=\|A(f^*)\|_{L^{n/\alpha,\infty}(\mathbb R^+)}\\
&\lesssim\|f^*\|_{L^{n/\alpha,\infty}(\mathbb R^+)}=\|f\|_{L^{n/\alpha,\infty}(\mathbb R^n)}<\infty.
\end{align*}
To finish, if $f\in L^{n/\alpha,\infty}(\mathbb R^n)$, then using \eqref{mallessup} and \eqref{equivnorm}, we have that 
$$
\|\mathcal{M}_\alpha f\|_{L^\infty(\mathbb R^n)}\lesssim \|f\|_{L^{n/\alpha,\infty}(\mathbb R^n)},
$$
and hence $f\in X_\alpha$. Thus, $L^{n/\alpha,\infty}(\mathbb R^n)$ is an r.i. space contained in $X_\alpha$ and, therefore,  $L^{n/\alpha,\infty}(\mathbb R^n)\subset \widetilde{X_\alpha}$. 

Conversely, if $f\in \widetilde{X_\alpha}$, consider $g$ as in \eqref{suplesmal}. Then $g^*=f^*$, and hence $g\in \widetilde{X_\alpha}$. Thus, $\mathcal{M}_\alpha g\in L^\infty(\mathbb R^n)$ and, using both \eqref{mallessup} and \eqref{suplesmal},
$$
\Big\|\sup_{t<s<\infty}s^{\alpha/n}f^{**}(s)\Big\|_{L^\infty(\mathbb R^+)}\simeq \|\mathcal{M}_\alpha g\|_{L^\infty(\mathbb R^n)} <\infty.
$$
Finally, \eqref{equivnorm} gives us that $f\in L^{n/\alpha,\infty}(\mathbb R^n)$.\end{proof}

\section{Proofs of Theorems \ref{thm:SW:Ainf:p:q:alpha0=0} and \ref{thm:SW:Ainf:p:q}}\label{secc:proof:main:thms}

\subsection{Proof of Theorem~\ref{thm:SW:Ainf:p:q:alpha0=0}} Fix $u, v^{-p'/p} \in \Ainf$. By Remark \ref{rmk:N:weights} with $N =2$, fix $\eps \in (0, 1/2)$ and given a cube $Q$, let $H_\eps^1$ and $H_\eps^2$ denote the subset $H_\eps \subset Q$ from Theorem~\ref{thm:Ainfty:H} \eqref{H:every:eps}, corresponding to $u$ and $v^{-p'/p}$, respectively. Set $H:= H_\eps^1 \cap H_\eps^2$, so that \eqref{meas:H:Q} yields $|H| \geq (1- 2\eps) |Q| >0$, and then \eqref{w:like:ave:Q} applied to  $u$ and $v^{-p'/p}$ gives
\begin{equation}\label{u:r:H}
\left(\fint_Q u  \right)^{1/q} \simeq u(x)^{1/q}, \quad \text{a.e. } x\in H
\end{equation}
as well as
\begin{equation}\label{vrp:r:H}
\left(\fint_Q v^{-p'/p}  \right)^{1/p'} \simeq v(x)^{-1/p}, \quad \text{a.e. } x\in H.
\end{equation}
Then, by fixing $x \in H$, \eqref{u:r:H} and \eqref{vrp:r:H} along with the hypothesis $u^{1/q} \leq C v^{1/p}$ a.e. $\rn$ yield
$$
 \left(\fint_Q u  \right)^{1/q}\left(\fint_Q  v^{-p'/p}\right)^{1/p'} \simeq  u(x)^{1/q} v(x)^{-1/p}  \lesssim C,
$$
where the implied constants depend only on $[u]_{\Ainf}$, $[v^{-p'/p}]_{\Ainf}$, $\eps$, $p$, and $n$. \qed

\subsection{Proof of Theorem~\ref{thm:SW:Ainf:p:q}} In keeping with the notation from the proof of Theorem~\ref{thm:SW:Ainf:p:q:alpha0=0},  from \eqref{u:r:H}, \eqref{vrp:r:H}, and the fact that $|H| \geq (1- 2\eps) |Q| >0$, given $1 < p \leq q < \infty$ we have 
\begin{equation*}
\left(\fint_Q u  \right)^{1/q}\left(\fint_Q  v^{-p'/p}\right)^{1/p'} \simeq \frac{u(x)^{1/q}}{v(x)^{1/p}}.
\end{equation*}
Now, since \eqref{u:r:H} and \eqref{vrp:r:H} imply $u(x) \simeq u(y)$ as well as $v(x) \simeq v(y)$ for a.e. $x, y \in H$, it follows that
$$
\frac{u(x)^{1/q}}{v(x)^{1/p}} \simeq \fint_H \frac{u^{1/q}}{v^{1/p}},  \quad \text{a.e. } x \in H.
$$
Consequently, 
\begin{equation}\label{u/v:p:q}
\left(\fint_Q u  \right)^{1/q}\left(\fint_Q  v^{-p'/p}\right)^{1/p'} \simeq \fint_H \frac{u^{1/q}}{v^{1/p}} \lesssim \fint_Q \frac{u^{1/q}}{v^{1/p}}
\end{equation}
which, from the definition of $\Theta(\alpha)$ in \eqref{def:alpha0}, yields 
\begin{equation}\label{SW:p:q:Ma0}
|Q|^{\alpha/n + 1/q - 1/p} \left(\fint_Q u  \right)^{1/q}\left(\fint_Q  v^{-p'/p}\right)^{1/p'} \lesssim |Q|^{\Theta(\alpha)/n} \fint_Q \frac{u^{1/q}}{v^{1/p}}.
\end{equation}
Finally, the hypothesis $u^{1/q}/v^{1/p} \in L^{\frac{n}{\Theta(\alpha)}, \infty}(\rn)$ and Theorem~\ref{thm:Ma:Linfty} imply 
$$
|Q|^{\Theta(\alpha)/n} \fint_Q \frac{u^{1/q}}{v^{1/p}} \leq \norm{\mathcal{M}_{\Theta(\alpha)}(u^{1/q}/v^{1/p})}{L^\infty(\rn)} \leq C \norm{u^{1/q}/v^{1/p}}{L^{n/\Theta(\alpha), \infty}(\rn)}
$$
and \eqref{SW:p:q} follows. If, in addition, $u^{1/q}/v^{1/p} \in \Ainf$, then fix $\eps \in (0, 1/3)$ and given a cube $Q$, let $H_\eps^1$, $H_\eps^2$, and $H_\eps^3$ denote the subset $H_\eps \subset Q$ from Theorem~\ref{thm:Ainfty:H} \eqref{H:every:eps}, corresponding to $u$, $v^{-p'/p}$, and $u^{1/q}/v^{1/p}$,  respectively, and put $H:= H_\eps^1 \cap H_\eps^2 \cap H_\eps^3$, so that \eqref{meas:H:Q} yields $|H| \geq (1- 3\eps) |Q| >0$. Then, the inequality $\fint_H \frac{u^{1/q}}{v^{1/p}} \lesssim \fint_Q \frac{u^{1/q}}{v^{1/p}}$ from \eqref{u/v:p:q} becomes the equivalence  $\fint_H \frac{u^{1/q}}{v^{1/p}} \simeq \fint_Q \frac{u^{1/q}}{v^{1/p}}$, which then turns \eqref{SW:p:q:Ma0} into 
$$
|Q|^{\alpha/n + 1/q - 1/p} \left(\fint_Q u  \right)^{1/q}\left(\fint_Q  v^{-p'/p}\right)^{1/p'} \simeq |Q|^{\Theta(\alpha)/n} \fint_Q \frac{u^{1/q}}{v^{1/p}}
$$
and therefore \eqref{SW:p:q} is equivalent to $\mathcal{M}_{\Theta(\alpha)}(u^{1/q}/v^{1/p}) \in L^\infty(\rn)$. \qed

\section{Examples of $(u,v)$ with $u= \mathcal{M}(g)^{\Theta(\alpha)q/n} v^{q/p}$ and $g \in L^1(\rn)$}\label{secc:examples:u:Mg} 

For $1 < p < \infty$ and $0 < \alpha <n$, in \cite[Section~3]{Perez1} C. P\'erez introduced the class $A_{p,\alpha}$ of pairs $(u,v)$ satisfying
\begin{equation*}
[(u,v)]_{A_{p,\alpha}}:= \sup\limits_Q |Q|^{\alpha /n}\left(\fint_Q u  \right)^{1/p} \left(\fint_Q  v^{-p'/p}\right)^{1/p'} < \infty
\end{equation*}
as well as the classes
\begin{align*}
B_{p,\alpha}:=&\{(u,v) \in A_{p,\alpha}: v^{-p'/p} \in \Ainf \},\\ 
C_{p,\alpha}:=&\{(u,v) \in A_{p,\alpha}: u \in \Ainf \},\\
D_\alpha:= &\left\{u \in \Ainf: \sup\limits_Q |Q|^{\alpha/n}\fint_Q u < \infty \right\}.
\end{align*}
Notice that weight pairs in the intersection $B_{p,\alpha} \cap C_{p,\alpha}$ satisfy \eqref{SW:p:q} (with $p=q$). As proved in \cite[Section~3]{Perez1}, if $0 < \alpha < n$, $1 < p< \infty$, and $v^{-p'/p} \in \Ainf$,  then $(\mathcal{M}_{\alpha p'}(v^{-p'/p}), v) \in B_{p,\alpha} \cap C_{p,\alpha}$ as long as $\mathcal{M}_{\alpha p'}(v^{-p'/p})$ is finite a.e. Similarly, if $\alpha p < n$, then $(u, \mathcal{M}_{\alpha p}(u)) \in A_{p,\alpha}$, for a general weight $u$ with $\mathcal{M}_{\alpha p}(u)$ finite a.e. Notice however that if $w \equiv 1$, then $\mathcal{M}_\alpha(w) \equiv \infty$ for $0 < \alpha < n$. In this section we will construct examples of weight pairs $(u,v) \in B_{p,\alpha} \cap C_{p,\alpha}$ based on the Hardy-Littlewood maximal operator $\mathcal{M}$.

Given $1 < p \leq q < \infty$,  $u, v^{-p'/p} \in \Ainf$, and $0 < \alpha < n$ with $\Theta(\alpha) > 0$, by Theorem~\ref{thm:SW:Ainf:p:q} the condition $u^{1/q}/v^{1/p} \in L^{{n/\Theta(\alpha)}, \infty}(\rn)$ implies \eqref{SW:p:q}. In this section we seek a pair $(u,v)$ such that
\begin{equation}\label{def:u:Mg:v:p:q}
u:= \mathcal{M}(g)^{\Theta(\alpha)q/n} v^{q/p},
\end{equation}
which immediately yields $u^{1/q}/v^{1/p} = \mathcal{M}(g)^{\Theta(\alpha)/n} \in L^{n/\Theta(\alpha), \infty}(\rn)$ for any $g \in L^1(\rn)$ with the estimate
\begin{equation}\label{uq:vp:alpha0:infty:M}
\norm{u^{1/q}/v^{1/p}}{L^{n/\Theta(\alpha), \infty}(\rn)} = \norm{\mathcal{M}(g)^{ \Theta(\alpha)/n} }{L^{n/\Theta(\alpha), \infty}(\rn)} \lesssim \norm{g}{L^1(\rn)}^{\Theta(\alpha)/n}.
\end{equation}

In what follows we are then only to find sufficient conditions for $u, v^{-p'/p} \in \Ainf$.

\begin{thm}\label{thm:Mg:p<q:u:A1} Fix $1 < p\leq q < \infty $ and $0 < \alpha <n$ with $\Theta(\alpha) :=\alpha + n (1/q-1/p) > 0$ and $p' \Theta(\alpha) < n$. Then, given $u \in A_1(\rn)$ and $g \in L^1(\rn)$ the pair 
$$
(u,v) = (u,  \mathcal{M}(g)^{-\Theta(\alpha) p/n} u^{p/q})
$$ 
satisfies \eqref{SW:p:q} with the estimate 
\begin{equation}\label{uv:g:alpha0/n}
[(u,v)]_{A_{p,q}^{\alpha}} \lesssim \norm{g}{L^1(\rn)}^{\Theta(\alpha)/n},
\end{equation}
where the implied constants depend only on $n$, $\alpha$, $p$, $q$, and $[u]_{A_1(\rn)}$. In particular, from \eqref{p:q:I:alpha} we obtain the weighted inequality
$$
\left(\int_{\rn} |I_\alpha(f)(x)|^q  u(x) \dx \right)^\frac{1}{q}\leq C \norm{g}{L^1(\rn)}^{\Theta(\alpha)/n} \left(\int_{\rn} |f(x)|^p \mathcal{M}(g)(x)^{-\Theta(\alpha)p/n} u(x)^{p/q} \dx \right)^\frac{1}{p},
$$
for every $g \in L^1(\rn)$ and measurable $f$, where $C > 0$ depends only on $n$, $\alpha$, $p$, $q$, and $[u]_{A_1(\rn)}$.
\end{thm}

\begin{proof} Let us write \eqref{def:u:Mg:v:p:q} as
$$
v^{-p'/p}= \mathcal{M}(g)^{\Theta(\alpha) p'/n} u^{-p'/q} =  \mathcal{M}(g)^{\Theta(\alpha) p'/n} u^{- (\tau -1)}
$$
with $\tau:= p'/q + 1$. Since we have $\mathcal{M}(g)^{\Theta(\alpha) p'/n} \in A_1(\rn)$, due to \eqref{Mg:theta:A1} with $\theta := \Theta(\alpha) p'/n \in (0,1)$, as well as $u \in  A_1(\rn)$ by hypothesis, it follows from \eqref{w1w2tau:A1} that $v^{-p'/p} \in A_{\tau}(\rn)$ . Thus,  by Theorem~\ref{thm:SW:Ainf:p:q}, the pair $(u,v)$ satisfies \eqref{SW:p:q} also with the estimate \eqref{uv:g:alpha0/n} due to \eqref{uv:pq:thm:SW} and  \eqref{uq:vp:alpha0:infty:M}.  \end{proof}

\begin{thm}\label{thm:Mg:p<q} Fix $0 < \alpha <n$ and  $1 < p  \leq q < n/\Theta(\alpha)$ with $\Theta(\alpha) :=\alpha + n (1/q-1/p) > 0$. Fix $v \in RH_\infty(\rn)$ such that $v^{-p'/p} \in \Ainf$ and $g \in L^1(\rn)$. Then the pair 
$$
(u,v)=(\mathcal{M}(g)^{\Theta(\alpha)q/n} v^{q/p}, v)
$$ 
satisfies \eqref{SW:p:q} with the estimate 
\begin{equation}\label{uv:prop:Mg:p<q}
[(u,v)]_{A_{p, q}^{\alpha}} \leq C_{\alpha, \delta, p, q, n} \norm{g}{L^1(\rn)}^{\Theta(\alpha)/n}.
\end{equation}
In particular, from \eqref{p:q:I:alpha} we obtain the weighted inequality
$$
\left(\int_{\rn} |I_\alpha(f)(x)|^q  \mathcal{M}(g)(x)^{\Theta(\alpha)q/n} v(x)^{q/p} \dx \right)^\frac{1}{q}\leq C \norm{g}{L^1(\rn)}^{\Theta(\alpha)/n} \left(\int_{\rn} |f(x)|^p v(x) \dx \right)^\frac{1}{p},
$$
for every $g \in L^1(\rn)$ and measurable $f$, where $C > 0$ depends only on $n$, $\alpha$, $p$, $q$, $[v^{-p/p}]_{\Ainf}$, and $[v]_{RH_\infty(\rn)}$.
\end{thm}

\begin{proof} First we show that $u :=\mathcal{M}(g)^{\Theta(\alpha)q/n} v^{q/p} \in \Ainf$. By assumption, $\Theta(\alpha)q < n$ and $g \in L^1(\rn)$, it then follows from \eqref{Mg:theta:A1} that $\mathcal{M}(g)(x)^{\Theta(\alpha)q/n} \in A_1(\rn)$. Next, the fact that $v \in RH_\infty(\rn)$ along with \eqref{RHinfty:powers} makes $v^{q/p} \in RH_\infty(\rn)$. In particular, $v^{q/p}$ is a multiplier of $\Ainf$ (recall \eqref{multipliers:Ainfty}), consequently $u \in \Ainf$ and all the hypotheses of Theorem~\eqref{thm:SW:Ainf:p:q} are met. Hence, $(u,v)$ satisfies \eqref{SW:p:q} with the estimate \eqref{uv:prop:Mg:p<q} due to \eqref{uv:pq:thm:SW} and  \eqref{uq:vp:alpha0:infty:M}. \end{proof}

For the case  $v(x):=|x|^\delta$ with $0 \leq \delta < n (p-1)$ Theorem~\ref{thm:Mg:p<q} yields

\begin{coro} Fix $0 < \alpha <n$ and  $1 < p  \leq q < n/\Theta(\alpha)$ with $\Theta(\alpha) :=\alpha + n (1/q-1/p) > 0$. Then, given $0 \leq \delta < n (p-1)$ we have
$$
\left(\int_{\rn} |I_\alpha(f)(x)|^q  \mathcal{M}(g)(x)^{\Theta(\alpha)q/n} |x|^{q\delta /p} \dx \right)^\frac{1}{q}\leq C \norm{g}{L^1(\rn)}^{\Theta(\alpha)/n} \left(\int_{\rn} |f(x)|^p |x|^\delta \dx \right)^\frac{1}{p},
$$
for every $g \in L^1(\rn)$ and measurable $f$, and $C > 0$ depends only on $n$, $\alpha$, $p$, $q$, and, $\delta.$
\end{coro}

\begin{proof} Set $v(x):=|x|^\delta$ (which belongs to $RH_\infty(\rn)$ due to Remark~\ref{comments:A1:RHinfty}) so that the condition $\delta < n (p-1)$ makes $|x|^{-\delta p'/p} \in \Ainf$.
\end{proof}
 
\section{Applications to weighted Hardy inequalities}\label{secc:app:Hardy} In \cite{CKN1},  L. Caffarelli, R. Kohn, and L. Nirenberg proved, among other results, that for $n >2$, $- \infty < a < \frac{n-2}{2}$, $a \leq b \leq a+1$, and $q = \frac{2n}{n-2+ 2(b-a)}$,  the weighted inequality 
\begin{equation}
\label{CKN:ineq:2:q}
\left(\int_{\rn}  |f(x)|^q |x|^{-bq} \, dx \right)^{1/q} \leq C \left(\int_{\rn} |\nabla f(x)|^2 |x|^{-2a}\, dx \right)^{1/2}, \quad \forall f \in C_c^1(\rn),
\end{equation}
for some $C=C(a,b,q,n) >0$. As an illustration of the applications of Theorems~\ref{thm:SW:Ainf:p:q:alpha0=0} and \ref{thm:SW:Ainf:p:q}, in this section we will establish new weighted Hardy inequalities, some of which will extend inequalities such as \eqref{CKN:ineq:2:q}, by utilizing the pairs $(u,v)$ developed in Section \ref{secc:examples:u:Mg}. Our starting point will be the pointwise estimate
\begin{equation}\label{f:I1:grad:f}
|f(x)| \leq C_n I_1(|\nabla f| )(x), \quad \forall x \in \rn,
\end{equation}
for every $f \in C_0^1(\rn)$ (i.e., $f$ continuously differentiable, vanishing at $\infty$), which is quickly proved, for given $f \in C_0^1(\rn)$, $x \in \rn$, and $\omega \in S^{n-1}$ the fundamental theorem of calculus gives
$$
f(x) = - \int_0^\infty \frac{d}{dr} f(x + r \omega) \, dr
$$ 
and then, after integration on both sides with respect to $\omega \in S^{n-1}$,
$$
f(x) =  - \frac{1}{n\omega_n} \int_0^\infty \int_{S^{n-1}} \frac{d}{dr} f(x + r \omega) \, dr d\omega  = \frac{1}{n\omega_n} \int_{\rn} \frac{(x- y) \cdot \nabla f(y)}{|x-y|^n} \, dy 
$$
and \eqref{f:I1:grad:f} follows. By combining \eqref{f:I1:grad:f} with  $\norm{I_1(h)}{L^q(u)} \leq C [(u,v)]_{A_{p,q}^{1}} \norm{h}{L^p(v)}$ from \eqref{p:q:I:alpha},  we obtain the Hardy-Sobolev inequality
\begin{equation}\label{HS:pq}
\norm{f}{L^q(u)} \leq C [(u,v)]_{A_{p,q}^{\alpha}} \norm{\nabla f}{L^p(v)}, \quad \forall f \in C_0^1(\rn) 
\end{equation}
whenever the pair $(u,v)$ satisfies the Sawyer-Wheeden condition \eqref{SW:p:q}. 

\begin{remark} Poincar\'e-type inequalities can be obtained as well by using that given a cube $Q \subset \rn$ the weighted inequality
\begin{equation}\label{Sob:u:v}
\left(\int_{Q} |f(x) - f_Q|^q u(x) \, dx \right)^{1/q} \leq C_S \left(\int_{Q} |\nabla f(x)|^p v(x) \, dx \right)^{1/p}
\end{equation}
holds true for every $f \in C^1(\rn)$ whenever the pair $(u,v)$ satisfies \eqref{SW:p:q} with $u, v^{-p'/p} \in \Ainf$. Moreover, in this case the constant $C_S > 0$ from \eqref{Sob:u:v} can be expressed as $C_S = C_{p,q,n} [(u,v)]_{A_{p,q}^{\alpha}}$, see for instance \cite[Theorem~5]{SW1}.
\end{remark}

\begin{remark} Although the applications below will focus on $I_1$, the method to obtain Hardy-type inequalities can be implemented for a fractional exponent $0 < \alpha < n$ by means of fractional derivatives $(-\Delta)^{\alpha/2}$ by replacing \eqref{f:I1:grad:f} with the identity
$$
f = I_\alpha((-\Delta)^{\alpha/2}(f))
$$
for every $f$ in the Schwartz class (see for instance \cite[p.117]{Stein}) which along with \eqref{p:q:I:alpha} yields
\begin{equation*}
\left(\int_{\rn} |f(x)|^q  u(x) \dx \right)^{1/q}\leq C  [(u,v)]_{A_{p,q}^{\alpha}}  \left(\int_{\rn} |(-\Delta)^{\alpha/2}f(x)|^p v(x) \dx \right)^{1/p}.
\end{equation*}
\end{remark}

\subsection{Hardy-type inequalities after Theorems~\ref{thm:SW:Ainf:p:q:alpha0=0} and \ref{thm:SW:Ainf:p:q}}\label{sec:gen:Hardy}

Before passing to specific choices of pairs $(u, v)$, let us state here the general forms of the two-weight Hardy-type inequalities that follow from Theorems~\ref{thm:SW:Ainf:p:q:alpha0=0} and \ref{thm:SW:Ainf:p:q}. Recall that, for $1 < p \leq q < \infty$, we have defined
$$
\Theta(1) := 1 + n \left(\frac{1}{q} - \frac{1}{p}\right).
$$
For the cases  $\Theta(1) =0$ and $\Theta(1) >0$ (corresponding to the cases $1/q = 1/p-1/n$ and $1/q > 1/p-1/n$, respectively) we have the following Theorems~\ref{thm:SW:Ainf:p:q:alpha0=0:app} and \ref{thm:SW:Ainf:p:q:alpha0>0:app}.

\begin{thm}\label{thm:SW:Ainf:p:q:alpha0=0:app} Fix $n > 1$ and $1 < p \leq q < \infty$ such that $\Theta(1)=0$. Given weights $u, v^{-p'/p} \in \Ainf$ such that $u(x)^{1/q} \leq C_1 v(x)^{1/p}$ for some $C_1 > 0$ and a.e. $x \in \rn$, we have 
\begin{equation*}
\left(\int_{\rn} |f(x)|^q  u(x) \dx \right)^{1/q}\leq C C_1 \left(\int_{\rn} |\nabla f(x)|^p v(x) \dx \right)^{1/p}, \quad \forall f \in C_0^1(\rn),
\end{equation*}
where $C > 0$ depends only on $n$, $p$, $q$, $[u]_{\Ainf}$, and $[v^{-p'/p}]_{\Ainf}$.
\end{thm}

\begin{proof} It follows from Theorem~\ref{thm:SW:Ainf:p:q:alpha0=0} with $\alpha =1$. 
\end{proof}

\begin{thm}\label{thm:SW:Ainf:p:q:alpha0>0:app} Fix $n > 1$ and $1 < p \leq q < \infty$ such that $\Theta(1) > 0$. Given weights $u, v^{-p'/p} \in \Ainf$ such that  $u^{1/q}/v^{1/p} \in L^{n/\Theta(1), \infty}(\rn)$, we have, for every $f \in C_0^1(\rn)$, 
$$
\left(\int_{\rn} |f(x)|^q  u(x) \dx \right)^{1/q} \leq C_2 \norm{u^{1/q}/v^{1/p}}{L^{n/\Theta(\alpha), \infty}(\rn)} \left(\int_{\rn} |f(x)|^p v(x) \dx \right)^{1/p},
$$
where $C_2 > 0$ depends only on $n$, $\alpha$, $p$, $q$, $[u]_{\Ainf}$, and $[v^{-p'/p}]_{\Ainf}$.
\end{thm}

\begin{proof} It follows from Theorem~\ref{thm:SW:Ainf:p:q} with $\alpha =1$. 
\end{proof}

\subsection{Hardy-type inequalities with $u= \mathcal{M}(g)^{\Theta(\alpha)q/n} v^{q/p}$ and $g \in L^1(\rn)$}

In this section we express the inequality \eqref{HS:pq} in terms of the weights developed in Section~\ref{secc:examples:u:Mg} and $\alpha =1$. 

\begin{thm}\label{thm:Mg:p<q:u:A1:app} Fix $n > 1$, $1 < p\leq q < \infty $ with $\Theta(1) :=1 + n (1/q-1/p) > 0$ and $\Theta(1) p' < n$. Then, given $u \in A_1(\rn)$ there exists $C > 0$ depending only on $n$, $p$, $q$, and $[u]_{A_1(\rn)}$ such that
$$
\left(\int_{\rn} |f(x)|^q  u(x) \dx \right)^{1/q}\leq C \norm{g}{L^1(\rn)}^{\Theta(1)/n} \left(\int_{\rn} |\nabla f(x)|^p \mathcal{M}(g)(x)^{- \Theta(1)p/n} u(x)^{p/q} \dx \right)^{1/p} 
$$
for every $f \in C_0^1(\rn)$ and $g \in L^1(\rn)$.
\end{thm}

\begin{proof} It follows from Theorem~\ref{thm:Mg:p<q:u:A1} with $\alpha =1$. 
\end{proof}

\begin{coro}\label{coro:Mg:p<q:u:A1:app} Fix $n > 1$, $1 < p\leq q < \infty $ with $\Theta(1) :=1 + n (1/q-1/p) > 0$ and $p' \Theta(1) < n$. Then, given $u \in A_1(\rn)$ there exists $C > 0$ depending only on $n$, $p$, $q$, and $[u]_{A_1(\rn)}$ such that
$$
\left(\int_{\rn} |f(x)|^q  u(x) \dx \right)^{1/q}\leq C  \left(\int_{\rn} |\nabla f(x)|^p |x|^{\Theta(1)p} u(x)^{p/q} \dx \right)^{1/p},  \quad \forall f \in C_0^1(\rn). 
$$
\end{coro}

\begin{proof}  For $k \in \na$ apply Theorem~\ref{thm:Mg:p<q:u:A1:app} to  $g_k := k^n \chi_{B(0, 1/k)}$, which then gives $ \norm{g_k}{L^1(\rn)} = C_n$ as well as $\mathcal{M}(g_k)(x) \geq |x|^{-n}$  (and hence $\mathcal{M}(g_k)(x)^{- \Theta(1)p/n} \leq |x|^{\Theta(1)p}$) for every $x \in \rn$ with $|x| > 1/k$. The corollary then follows by letting $k \to \infty$.
\end{proof}

\begin{remark} When $n >2$ one recovers the Caffarelli-Kohn-Nirenberg inequality \eqref{CKN:ineq:2:q} from Corollary \ref{coro:Mg:p<q:u:A1:app} by choosing $p=2$ and, given $q$ and $a,b$, with $a < b$, as in \eqref{CKN:ineq:2:q}, by setting $u(x):=|x|^{-bq}$. Then, $u \in A_1(\rn)$ if and only if $bq < n$, which means
$$
bq = \frac{2nb}{n-2 + 2 (b-a)} < n
$$
and it amounts to $a < (n-2)/2$. Also, from the definition of $q = \frac{2n}{n-2+ 2(b-a)}$ we get $\Theta(1) = b-a$ and then $|x|^{p \Theta(1)} u(x)^{p/q} =|x|^{2(\Theta(1) - b)}= |x|^{-2a}$. In particular, $(b-a) = \Theta(1) > 0$ and $\Theta(1) \leq 1$ (because $q \geq p=2)$.  Finally, $ \Theta(1)p' < n$ means $2 \Theta(1) < n$ which holds true because $0 < \Theta(1) \leq 1$ and $n >2$.
\end{remark}

\begin{thm}\label{thm:Mg:p<q:app} Fix $n > 1$ and $1 < p  \leq q < n/\Theta(1)$ with $\Theta(1) :=1 + n (1/q-1/p) > 0$. Given $v \in RH_\infty(\rn)$ such that $v^{-p'/p} \in \Ainf$  there exists $ C >0$ depending only on $n$, $p$, $q$, $[v^{-p/p}]_{\Ainf}$, and $[v]_{RH_\infty(\rn)}$ such that
$$
\left(\int_{\rn} |f(x)|^q  \mathcal{M}(g)(x)^{\Theta(1)q/n} v(x)^{q/p} \dx \right)^{1/q}\leq C \norm{g}{L^1(\rn)}^{\Theta(1)/n} \left(\int_{\rn} |\nabla f(x)|^p v(x) \dx \right)^{1/p}
$$
for every $f \in C_0^1(\rn)$ and $g \in L^1(\rn)$. 
\end{thm}

\begin{proof} It follows from Theorem~\ref{thm:Mg:p<q} with $\alpha =1$. 
\end{proof}

\begin{coro}\label{thm:Hardy:ineq:p:q:M}  Fix $n > 1$ and $1 < p  \leq q < n/\Theta(1)$ with $\Theta(1) :=1 + n (1/q-1/p) > 0$. Given $0 \leq \delta < n(p-1)$ there exists $C > 0$, depending only on $\delta$, $p$, $q$, and $n$, such that
\begin{equation}\label{Hardy:ineq:p:p:M}
\left(\int_{\rn} |f(x)|^q  \mathcal{M}(g)(x)^{\Theta(1)q/n} |x|^{q\delta /p} \dx \right)^{1/q}\leq C \norm{g}{L^1(\rn)}^{\Theta(1)/n} \left(\int_{\rn} |\nabla f(x)|^p |x|^\delta \dx \right)^{1/p}
\end{equation}
for every $f \in C_0^1(\rn)$ and $g \in L^1(\rn)$. 
\end{coro}

\begin{proof} The inequality \eqref{Hardy:ineq:p:p:M} follows from Theorem~\ref{thm:Mg:p<q:app} by taking $v(x)=|x|^\delta$.
\end{proof}

\end{document}